\documentclass[a4paper,11pt]{article}
\usepackage[latin1]{inputenc}
\usepackage[english]{babel}
\usepackage{amsmath}
\usepackage{amsfonts}
\usepackage{amssymb}
\usepackage{epsfig}
\usepackage{amsopn}
\usepackage{amsthm}
\usepackage{color}
\usepackage{graphicx}
\usepackage{enumerate}
\usepackage{mathrsfs}
\usepackage{cite}
\newtheorem{theorem}{Theorem}[section]
\newtheorem{corollary}[theorem]{Corollary}
\newtheorem{lemma}[theorem]{Lemma}
\newtheorem{proposition}[theorem]{Proposition}

\newtheorem*{theorem*}{Theorem}
\newtheorem*{lemma*}{Lemma}
\newtheorem*{remark*}{Remark}
\newtheorem*{definition*}{Definition}
\newtheorem*{proposition*}{Proposition}
\newtheorem*{corollary*}{Corollary}
\numberwithin{equation}{section}
\newcommand{\real}{\mathbb{R}}

\let\ced=\c         

\def\qed{\,\unskip\kern 6pt \penalty 500
\raise -2pt\hbox{\vrule \vbox to8pt{\hrule width 6pt
\vfill\hrule}\vrule}\par}
\definecolor{darkblue}{rgb}{0.05, .05, .65}
\definecolor{darkgreen}{rgb}{0.1, .65, .1}
\definecolor{darkred}{rgb}{0.8,0,0}
\newcommand{\beqn}{\begin{equation}}
\newcommand{\eeqn}{\end{equation}}
\newcommand{\bear}{\begin{eqnarray}}
\newcommand{\eear}{\end{eqnarray}}
\newcommand{\bean}{\begin{eqnarray*}}
\newcommand{\eean}{\end{eqnarray*}}
\begin{document}

\title{\huge \bf Eternal solutions to a porous medium equation with strong nonhomogeneous absorption. Part I: Radially non-increasing profiles}

\author{
\Large Razvan Gabriel Iagar\,\footnote{Departamento de Matem\'{a}tica
Aplicada, Ciencia e Ingenieria de los Materiales y Tecnologia
Electr\'onica, Universidad Rey Juan Carlos, M\'{o}stoles,
28933, Madrid, Spain, \textit{e-mail:} razvan.iagar@urjc.es}
\\[4pt] \Large Philippe Lauren\ced{c}ot\,\footnote{Laboratoire de Math\'ematiques (LAMA) UMR~5127, Universit\'e Savoie Mont Blanc, CNRS, F--73000 Chamb\'ery, France. \textit{e-mail:} philippe.laurencot@univ-smb.fr}\\
\\[4pt]}
\date{}
\maketitle

\begin{abstract}
Existence of specific \emph{eternal solutions} in exponential self-similar form to the following quasilinear diffusion equation with strong absorption
$$
\partial_t u=\Delta u^m-|x|^{\sigma}u^q,
$$
posed for $(t,x)\in(0,\infty)\times\real^N$, with $m>1$, $q\in(0,1)$ and $\sigma=\sigma_c:=2(1-q)/(m-1)$ is proved. Looking for radially symmetric solutions of the form
$$
u(t,x)=e^{-\alpha t}f(|x|e^{\beta t}), \qquad \alpha=\frac{2}{m-1}\beta,
$$
we show that there exists a unique exponent $\beta^*\in(0,\infty)$ for which there exists a one-parameter family $(u_A)_{A>0}$ of solutions with compactly supported and non-increasing profiles $(f_A)_{A>0}$ satisfying $f_A(0)=A$ and $f_A'(0)=0$. An important feature of these solutions is that they are bounded and do not vanish in finite time, a phenomenon which is known to take place for all non-negative bounded solutions when $\sigma\in (0,\sigma_c)$.
\end{abstract}

\smallskip


\smallskip


\section{Introduction and main results}

The goal of the present paper (and also of its second part \cite{ILS23}) is to address the problem of existence and classification of some specific solutions to the following porous medium equation with strong absorption
\begin{equation}\label{eq1}
\partial_t u-\Delta u^m+|x|^{\sigma}u^q=0, \qquad (t,x)\in(0,\infty)\times\real^N,
\end{equation}
in the range of exponents
\begin{equation}\label{range.exp}
m>1, \qquad q\in(0,1), \qquad \sigma=\sigma_c :=\frac{2(1-q)}{m-1}.
\end{equation}
On the one hand, Eq.~\eqref{eq1} features, in the range of exponents given in \eqref{range.exp}, a competition between the degenerate diffusion term, which tends to conserve the total mass of the solutions while expanding their supports, and the absorption term which leads to a loss of mass. As it has been established and will be explained below, absorption becomes stronger as its exponent $q$ decreases and dominant in the range we are dealing with, leading to specific, although sometimes surprising phenomena such as finite time extinction, instantaneous shrinking and localization of the supports of the solutions. On the other hand, the weight $|x|^{\sigma}$ with $\sigma>0$ affects the absorption in the sense of enhancing its effect over regions far away from the origin, where $|x|$ is large, while reducing its strength near $x=0$, where $|x|^{\sigma}$ is almost zero (and formally there is even no absorption at $x=0$).

The balance between these two effects has been best understood in the spatially homogeneous case $\sigma=0$ of Eq.~\eqref{eq1}. A lot of development has been done several decades ago in the range $q>m>1$ where the diffusion is strong and the absorption is not leading the dynamics of the equations, see for example \cite{KP86, PT86, KU87, KV88, KPV89, Le97, Kwak98} and references therein. In this range, the previous knowledge of the porous medium equation and its self-similar behavior had a strong influence in developing the theory. The intermediate range $1<q\leq m$ is not yet totally understood in higher space dimensions. In dimension $N=1$ it has been shown that solutions are global in time but their supports \emph{are localized} if the initial condition is compactly supported; that is, there exists a radius $R>0$ not depending on time such that ${\rm supp}\,u(t)\subseteq B(0,R)$ for any $t>0$. Self-similar solutions might become unbounded \cite{MPV91, CVW97} and thus a delicate analysis of the large time behavior, involving the formation of boundary layers, is needed, see \cite{CV99}. Such descriptions are still lacking in dimension $N\geq2$.

More related to our study, still assuming that $\sigma=0$, the range $q\in(0,1)$ is the most striking one, where the absorption term dominates  the diffusion and leads to two new mathematical phenomena. On the one hand, the \emph{finite time extinction} stemming from the ordinary differential equation $\partial_t u=-u^q$ obtained by neglecting the diffusion has been established by Kalashnikov \cite{Ka74, Ka84}, emphasizing the dominance of the absorption term. On the other hand, \emph{instantaneous shrinking of supports} of solutions to Eq.~\eqref{eq1} (with $\sigma=0$) emanating from a bounded initial condition $u_0$ such that $u_0(x)\to 0$ as $|x|\to\infty$ takes place, that is, for any non-negative initial condition $u_0\in L^{\infty}(\real^N)$ such that $u_0(x)\to0$ as $|x|\to\infty$ and $\tau>0$, there is $R(\tau)>0$ such that ${\rm supp}\,u(t)\subseteq B(0,R(\tau))$ for all $t\ge\tau$. This rather unexpected behavior is once more due to the strength of the absorption, which involves a very quick loss of mass and has been proved in \cite{Abd98} after borrowing ideas from previous works \cite{EK79, Ka84} devoted to the semilinear case. Finer properties of the dynamics of Eq.~\eqref{eq1} for $\sigma=0$ in this range, such as the behavior near the extinction time or even the extinction rates, are still lacking in a number of cases and seem (up to our knowledge) to be available only when $m+q=2$ in \cite{GV94}, revealing a case of asymptotic simplification. Completing this picture with the cases when $m+q\neq2$ appears to be a rather complicated open problem.

Drawing our attention now to the spatially inhomogeneous Eq.~\eqref{eq1} when $\sigma>0$, recent results have shown that the magnitude of $\sigma$ has a very strong influence on the dynamics of Eq.~\eqref{eq1} and, in some cases, the weight actually allows for a better understanding of the dynamics. More precisely, the analysis performed by Belaud and collaborators \cite{Belaud01, BD10, BeSh22}, along with the instantaneous shrinking of supports for bounded solutions to Eq.~\eqref{eq1} proved in \cite{ILS22}, shows that, for $0<\sigma<\sigma_c$, \emph{any non-negative solution} to Eq.~\eqref{eq1} with bounded initial condition \emph{vanishes in finite time}. A more direct proof of this result is given by the authors in the recent short note \cite{IL22}. On the contrary, after developing the general theory of well-posedness for Eq.~\eqref{eq1}, we have focused on the range $\sigma>\sigma_c$ in our previous work \cite{ILS22} and proved that, in the latter, finite time extinction \emph{depends strongly on how concentrated} is the initial condition in a neighborhood of the origin. More precisely, initial data which are positive in a ball $B(0,\delta)$ give rise to solutions with a non-empty positivity set for all times,
	\begin{equation}
		\{ x\in\mathbb{R}^N\ :\ u(t,x)>0\} \ne \emptyset \;\;\text{for all}\;\; t>0, \label{ppos}
	\end{equation}
when $\sigma>\sigma_c$, while initial data which vanish in a suitable way as $x\to 0$ and with a sufficiently small $L^{\infty}$ norm lead to solutions vanishing in finite time, as proved in \cite{IL22} where optimal conditions are given. All these cases of different dynamics are consequences of the two types of competitions explained in the previous paragraphs.

The exponent $\sigma_c=2(1-q)/(m-1)$ thus appears to separate the onset of extinction in finite time for arbitrary non-negative and bounded initial conditions which occurs for lower values of $\sigma$ and the positivity property~\eqref{ppos} which is known to take place for higher values of $\sigma$, in particular for initial conditions which are positive in a ball $B(0,\delta)$. According to \cite{IL22}, when $\sigma=\sigma_c$, there are non-negative solutions to Eq.~\eqref{eq1} vanishing in finite time, their initial conditions having a sufficiently small $L^\infty$-norm and decaying to zero in a suitable way as $x\to 0$, and the issue we address here is whether the positivity property~\eqref{ppos} also holds true for some solutions to Eq.~\eqref{eq1} when $\sigma=\sigma_c$. We actually construct specific solutions to Eq.~\eqref{eq1} with $\sigma=\sigma_c$ featuring this property and these solutions turn out to have an exponential self-similar form as explained in detail below. In particular, they are defined for all $t\in\mathbb{R}$.

\medskip

\noindent \textbf{Main results}. We are looking in this paper for some special solutions to Eq.~\eqref{eq1} with $m$, $q$ and $\sigma=\sigma_c$ as in \eqref{range.exp} having an \emph{exponential self-similar form}, that is,
\begin{equation}\label{SSS}
u(t,x)=e^{-\alpha t}f(|x|e^{\beta t}), \qquad (t,x)\in(0,\infty)\times\real^N.
\end{equation}
Notice that solutions as in~\eqref{SSS} are actually defined for all $t\in(-\infty,\infty)$; that is, they are not only global in time but \emph{eternal}. Even if solutions of the form~\eqref{SSS} are rather unexpected for parabolic equations due to the irreversibility of time, several equations are known to have such solutions but usually in critical cases separating different behaviors. Parabolic equations featuring this property include the two-dimensional Ricci flow \cite{DS06, Ham1993}, the fast diffusion equation with critical exponent $m_c=(N-2)/N$ in space dimension $N\geq 3$ \cite{GPV00}, a viscous Hamilton-Jacobi equation featuring singular diffusion of $p$-Laplacian type, $p\in (2N/(N+1),2)$, and critical gradient absorption \cite{IL13}, and the related reaction-diffusion equation $\partial_t u - \Delta u^m - |x|^\sigma u^q=0$ \cite{IS22a, IS22b}. Concerning the latter, the critical value of $\sigma$ is exactly the same as in \eqref{range.exp}, but the dynamic properties of the solutions strongly differ from the present work, since the spatially inhomogeneous part is there a source term, introducing mass to the equation. Eternal solutions are also available for kinetic equations, such as the spatially homogeneous Boltzmann equation for Maxwell molecules \cite{BoCe2002, Cab1998} or Smoluchowski's coagulation equation with coagulation kernel of homogeneity one \cite{Ber2002, BNV2019}. Let us finally mention that, besides solutions of the form~\eqref{SSS}, another important class of self-similar eternal solutions of evolution problems is that of traveling wave solutions of the form $(t,x)\mapsto u(x-ct)$ in space dimension $N=1$, which are available for scalar conservation laws and parabolic equations such as the celebrated Fisher-KPP equation, see \cite{CaCo2003, GK2004, Ser2004} and the references therein.

Returning to the ansatz~\eqref{SSS}, setting $\xi=|x|e^{\beta t}$ and performing some direct calculations, we readily find that the self-similar exponents must satisfy the condition
\begin{equation}\label{SS.exp}
\alpha=\frac{2}{m-1}\beta,
\end{equation}
where $\beta$ becomes a free parameter for our analysis, while the profile $f$ solves the differential equation
\begin{equation}\label{ODE}
(f^m)''(\xi)+\frac{N-1}{\xi}(f^m)'(\xi)+\alpha f(\xi)-\beta\xi f'(\xi)-\xi^{\sigma}f^q(\xi)=0, \qquad \xi>0.
\end{equation}
The solutions to Eq.~\eqref{ODE} we are looking for in this first part of a two-parts work are solutions taking positive values at $\xi=0$. To this end, let us observe that we can fix, without loss of generality, the initial condition as
\begin{equation}\label{initcond1}
f(0)=1, \qquad f'(0)=0.
\end{equation}
Indeed, given $a>0$ and a solution $f$ to~\eqref{ODE}-\eqref{initcond1}, we can readily obtain by direct calculations that the rescaled function
\begin{equation}\label{resc}
g(\xi;a)=af(a^{-(m-1)/2}\xi)
\end{equation}
solves \eqref{ODE} with initial conditions $g(0;a)=a$, $g'(0;a)=0$. This leaves us with the task of solving the Cauchy problem \eqref{ODE}-\eqref{initcond1}, which is performed in the next result.

\begin{theorem}\label{th.1}
Let $m$, $q$ and $\sigma=\sigma_c$ as in \eqref{range.exp}. There exists a unique exponent $\beta^*$ (and corresponding $\alpha^*=2\beta^*/(m-1)$) such that, for $\alpha=\alpha^*$ and $\beta=\beta^*$, the Cauchy problem \eqref{ODE}-\eqref{initcond1} has a compactly supported and non-negative solution $f^*$. The function $U^*$ defined by
$$
U^*(t,x)=e^{-\alpha^*t}f^*(|x|e^{\beta^*t}), \qquad (t,x)\in\mathbb{R}\times \mathbb{R}^N,
$$
is then a self-similar solution to Eq.~\eqref{eq1} in exponential form~\eqref{SSS}.
\end{theorem}

Let us point out that, in strong contrast with the range $\sigma>2(1-q)/(m-1)$ analyzed in \cite{ILS22} and where the self-similarity exponents were uniquely determined, in the present case we have two free parameters for the shooting technique: both the initial value of the solution at $x=0$ and the self-similar exponent $\beta$. Thus, in order to have uniqueness, we need to fix this initial value in view of the rescaling~\eqref{resc}, as explained above.

One of the interesting features of this work is the fact that the proof of Theorem~\ref{th.1} is based on a \emph{mix between various techniques}. We employ mainly a shooting technique with respect to the free parameter $\beta$, but in order to study the interface behavior and establish the uniqueness in Theorem \ref{th.1} we transform \eqref{ODE} into a quadratic three-dimensional autonomous dynamical system and study a specific local behavior and critical point in the associated phase space. Let us stress here that we have to go deeper than the analogous study of the interface behavior in \cite[Section 4]{ILS22}, since in some cases we need a second order local expansion near the interface point.

We end up this presentation by mentioning that the present work is the first part of a two-parts analysis of eternal solutions to Eq.~\eqref{eq1} and will be followed by a companion work \cite{ILS23} in which a second and rather surprising type of profiles, presenting a dead-core, is identified and classified, by employing a quite different bunch of techniques based on the complete analysis of an auxiliary dynamical system. Altogether, the existence of such a variety of self-similar solutions in exponential form shows that the dynamics of Eq.~\eqref{eq1} in the critical case $\sigma=\sigma_c$ is expected to be rather complex and to depend on many features of the initial conditions (such as concentration near $x=0$, magnitude of $\|u_0\|_{\infty}$ and location of the points where the maximum is attained, to name but a few) and is definitely a challenging problem.

\section{Proof of Theorem~\ref{th.1}}\label{sec.th1}

The proof of Theorem~\ref{th.1} is based on a shooting method with respect to the free exponent $\beta$ and follows the same strategy as \cite[Section~4]{ILS22}. However, a number of preparatory results are proved in a different way and the analysis near the interface requires to be improved in some cases with the help of a phase space analysis. We divide this section into several subsections containing the main steps of the proof.

\subsection{Existence of a compactly supported self-similar solution}\label{subsec.exist}

Let $\beta>0$ and $\alpha = 2\beta/(m-1)$. Recalling the differential equation~\eqref{ODE} satisfied by the self-similar profiles $f$ and setting for simplicity $F=f^m$, we study, as explained in the Introduction, the Cauchy problem
\begin{subequations}\label{eq.exist}
\begin{equation}\label{SSODE}
F''(\xi)+\frac{N-1}{\xi}F'(\xi)+\alpha f(\xi)-\beta\xi f'(\xi)-\xi^{\sigma}f^q(\xi)=0,
\end{equation}
\begin{equation}\label{Cauchy.cond}
F(0)=1, \qquad F'(0)=0.
\end{equation}
\end{subequations}
We obtain from the Cauchy-Lipschitz theorem that the problem~\eqref{eq.exist} has a unique positive solution $F(\cdot;\beta)\in C^2([0,\xi_{\max}(\beta)))$ defined on a maximal existence interval for which we have the following alternative: either $\xi_{\max}(\beta)=\infty$ or
$$
\xi_{\max}(\beta)<\infty \qquad {\rm and} \qquad \lim\limits_{\xi\to\xi_{\max}(\beta)}\left[F(\xi;\beta)+\frac{1}{F(\xi;\beta)}\right]=\infty.
$$
We next define
\begin{equation}\label{support}
\xi_0(\beta):=\inf\{\xi\in(0,\xi_{\max}(\beta))\ :\ f(\xi)=0\}\in(0,\xi_{\max}(\beta)],
\end{equation}
and
\begin{equation}\label{decrease}
\xi_1(\beta):=\sup\{\xi\in(0,\xi_{0}(\beta))\ :\ f'<0 \ {\rm on} \ (0,\xi)\}.
\end{equation}
We readily notice from \eqref{SSODE} and the $C^2$-regularity of $F$ that
\begin{equation}\label{init.sec.der}
F''(0;\beta)=-\frac{2\beta}{(m-1)N}<0,
\end{equation}
so that $\xi_1(\beta)>0$. Let us now study more precisely the behavior of $F(\cdot;\beta)$ near $\xi_0(\beta)$ when $\xi_0(\beta)$ is finite.

\begin{lemma}\label{lem.ext}
	Consider $\beta>0$ such that $\xi_0(\beta)<\infty$. Then $\xi_{\max}(\beta)=\xi_0(\beta)$ and $F=F(\cdot;\beta)\in C^1([0,\xi_0(\beta)])$ satisfies $F(\xi_0(\beta))=0$ and
	\begin{equation*}
		F'(\xi_0(\beta))  = \xi_0(\beta)^{1-N} \int_0^{\xi_0(\beta)} \xi_*^{N-1} \big[ \xi_*^\sigma f^q(\xi_*) - (\alpha+N\beta) f(\xi_*) \big]\,d\xi_*,
	\end{equation*}
recalling that $f=F^{1/m}$. Furthermore, if $\xi_1(\beta)=\xi_0(\beta)$, then the extension of $F$ by zero on $(\xi_0(\beta),\infty)$ belongs to $C^2([0,\infty))$ and is a solution to~\eqref{eq.exist} on $[0,\infty)$ with
\begin{equation*}
	F(\xi_0(\beta)) = F'(\xi_0(\beta)) = F''(\xi_0(\beta)) =0.
\end{equation*}
\end{lemma}

\begin{proof}
	As $\xi_0(\beta)<\infty$, then the above alternative implies that $\xi_{\max}(\beta)=\xi_0(\beta)$ and
	\begin{equation}
		\lim_{\xi\to\xi_0(\beta)} F(\xi)=0. \label{ext0}
	\end{equation}
Moreover, it follows from~\eqref{SSODE} that
\begin{equation*}
	\frac{d}{d\xi} \big[ \xi^{N-1} F'(\xi) + \beta \xi^N f(\xi) \big] = \xi^{N-1} \big[ \xi^\sigma f^q(\xi) - (\alpha+N\beta) f(\xi) \big]
\end{equation*}
for $\xi\in [0,\xi_0(\beta))$; hence, after integration over $(0,\xi)$,
\begin{equation*}
	\xi^{N-1} F'(\xi) + \beta \xi^N f(\xi) = \int_0^\xi \xi_*^{N-1} \big[ \xi_*^\sigma f^q(\xi_*) - (\alpha+N\beta) f(\xi_*) \big]\,d\xi_*.
\end{equation*}
Since we have already established in~\eqref{ext0} that $F$ and $f$ have a continuous extension on $[0,\xi_0(\beta)]$, we may take the limit $\xi\to\xi_0(\beta)$ in the above identity and complete the proof of the first statement of Lemma~\ref{lem.ext}. Owing to the definition of $\xi_1(\beta)$, the second statement then readily follows with the help of~\eqref{SSODE}.
\end{proof}

We now introduce the following three sets:
\begin{equation*}
\begin{split}
&\mathcal{A}:=\{\beta>0:\xi_0(\beta)<\infty \ {\rm and} \ F'(\xi;\beta)<0 \ {\rm for} \ \xi\in(0,\xi_0(\beta)]\}, \\
&\mathcal{C}:=\{\beta>0:\xi_1(\beta)<\xi_0(\beta)\},\\
&\mathcal{B}:=(0,\infty)\setminus(\mathcal{A}\cup\mathcal{C}),
\end{split}
\end{equation*}
and observe that $\mathcal{A}\cap\mathrm{C}=\emptyset$. Let us first show that the sets $\mathcal{A}$ and $\mathcal{C}$ are non-empty and open.

\begin{lemma}\label{lem.A}
The set $\mathcal{A}$ is non-empty and open and there exists $\beta_u>0$ such that $(\beta_u,\infty)\subseteq\mathcal{A}$.
\end{lemma}

\begin{proof}
Set $g(\xi;\beta)=f(\xi/\sqrt{\beta};\beta)$ for $\xi\in [0,\sqrt{\beta} \xi_0(\beta)]$, or equivalently $f(\xi;\beta)=g(\xi\sqrt{\beta};\beta)$ for $\xi\in  [0,\xi_0(\beta)]$. Setting also $G:=g^m$, we obtain by straightforward calculations that $g$ (and thus $G$) solves the Cauchy problem
\begin{subequations}
\begin{equation}\label{SSODE.A}
G''(\zeta)+\frac{N-1}{\zeta}G'(\zeta)+\frac{2}{m-1}g(\zeta)-\zeta g'(\zeta)-\beta^{-(\sigma+2)/2}\zeta^{\sigma}g^q(\zeta)=0,
\end{equation}
\begin{equation}\label{Cauchy.cond.A}
G(0)=1, \qquad G'(0)=0,
\end{equation}
\end{subequations}
where $\zeta=\xi\sqrt{\beta}$. Noticing that in the limit $\beta\to\infty$ the last term in \eqref{SSODE.A} vanishes, we proceed exactly as in the proof of \cite[Lemma~4.4]{ILS22} (see also the proof of \cite[Theorem~2]{Shi04} from where the idea comes) to conclude that there exists $\beta_u>0$ such that $(\beta_u,\infty)\subseteq\mathcal{A}$. We omit here the details as they are totally similar to the ones in the quoted references. That $\mathcal{A}$ is open is an immediate consequence of the continuous dependence of $f(\cdot;\beta)$ on $\beta$.
\end{proof}

As for the set $\mathcal{C}$, we do not need a rescaling in order to prove that it is non-empty.

\begin{lemma}\label{lem.C}
The set $\mathcal{C}$ is non-empty and open and there exists $\beta_l>0$ such that $(0,\beta_l)\subseteq\mathcal{C}$.
\end{lemma}

\begin{proof}
We obtain by letting $\beta\to 0$ in \eqref{eq.exist} that the limit equation is
\begin{subequations}\label{eqC}
\begin{equation}\label{SSODE.C}
H''(\xi)+\frac{N-1}{\xi}H'(\xi)-\xi^{\sigma}H^{q/m}(\xi)=0,
\end{equation}
with initial conditions
\begin{equation}\label{Cauchy.cond.C}
H(0)=1, \qquad H'(0)=0.
\end{equation}
\end{subequations}
By the Cauchy-Lipschitz theorem, the problem \eqref{eqC} has a unique positive solution $H\in C^2([0,\xi_H))$ defined on a maximal existence interval for which we have the following alternative: either $\xi_H=\infty$ or
$$
\xi_H<\infty \qquad {\rm and} \qquad \lim\limits_{\xi\to\xi_H}\left[H(\xi)+\frac{1}{H(\xi)}\right]=\infty.
$$
It follows from \eqref{eqC} that
$$
\frac{d}{d\xi}(\xi^{N-1}H'(\xi))=\xi^{N-1}\left[H''(\xi)+\frac{N-1}{\xi}H'(\xi)\right]=\xi^{N+\sigma-1}H^{q/m}(\xi)>0.
$$
Hence $\xi^{N-1}H'(\xi)>0$ and thus $H'(\xi)>0$ for any $\xi\in (0,\xi_H)$. Given $\delta\in (0,\xi_H)$ fixed, we have $H'(\delta)>0$ and $H(\xi)>1$ for any $\xi\in(0,\delta)$. The continuous dependence with respect to the parameter $\beta$ in \eqref{eq.exist} ensures that there exists $\beta_l>0$ such that
$$
F(\xi;\beta)>\frac{1}{2}, \qquad \xi\in[0,\delta], \qquad F'(\delta;\beta)>\frac{H'(\delta)}{2}>0
$$
for any $\beta\in(0,\beta_l)$. Recalling \eqref{support} and \eqref{decrease}, we conclude that $\xi_1(\beta)\in(0,\delta)$ and $\xi_0(\beta)>\delta$ for any $\beta\in(0,\beta_l)$; that is, $\xi_1(\beta)<\xi_0(\beta)$ for $\beta\in (0,\beta_l)$and $(0,\beta_l)\subseteq\mathcal{C}$. We use once more the continuous dependence with respect to the parameter $\beta$ of $F(\cdot;\beta)$ to conclude that $\mathcal{C}$ is open.
\end{proof}

We infer from Lemmas~\ref{lem.A} and~\ref{lem.C} that the set $\mathcal{B}$ is non-empty and closed. The instantaneous shrinking of supports of bounded solutions to Eq.~\eqref{eq1} proved in \cite[Theorem~1.1]{ILS22}, together with the definition of the set $\mathcal{A}$, readily gives the following characterization of the elements in the set $\mathcal{B}$.

\begin{lemma}\label{lem.B}
Let $\beta\in\mathcal{B}$. Then $\xi_0(\beta)=\xi_1(\beta)<\infty$ and $(f^m)'(\xi_0(\beta);\beta)=0$.
\end{lemma}

The proof is immediate and is given with details in \cite[Lemma~4.6]{ILS22}. We thus conclude that, for any element $\beta\in\mathcal{B}$, we have an eternal self-similar solution to Eq.~\eqref{eq1} in the form~\eqref{SSS} with profile $f(\cdot;\beta)$ as in Lemma~\ref{lem.B}.

\subsection{Monotonicity}\label{subsec.monot}

In this section we prove the following general monotonicity property of the profiles $f(\cdot;\beta)$ solving~\eqref{eq.exist} with respect to the parameter $\beta$.

\begin{lemma}\label{lem.monot}
Let $0<\beta_1<\beta_2<\infty$. Then
$$
f(\xi;\beta_1)>f(\xi;\beta_2) \qquad {\rm for \ any } \qquad \xi\in \big(0,\min\{\xi_1(\beta_1), \xi_1(\beta_2)\}\big).
$$
\end{lemma}

\begin{proof}
Consider $\beta_2>\beta_1>0$ and pick $X\in \big(0,\min\{\xi_1(\beta_1), \xi_1(\beta_2)\}\big)$. Then
$$
F_i:=F(\cdot;\beta_i)>0, \qquad F_i'<0, \qquad {\rm in} \ (0,X).
$$
Since $\beta_2>\beta_1$ and $F_1(0)=F_2(0)=1$, $F_1'(0)=F_2'(0)=0$, we infer from \eqref{init.sec.der} that $F_2<F_1$ in a right-neighborhood of $\xi=0$. We may thus define
$$
\xi_*:=\inf\{\xi\in(0,X):F_1(\xi)=F_2(\xi)\}>0,
$$
and notice that $F_2(\xi)<F_1(\xi)$ for any $\xi\in(0,\xi_*)$. Assume for contradiction that $\xi_*<X$. Then $F_2(\xi_*)=F_1(\xi_*)$. We introduce for any $\lambda\geq1$ the following family of rescaled functions
\begin{equation}\label{resc2}
G_{\lambda}(\xi):=\lambda^mF_2(\lambda^{-(m-1)/2}\xi), \qquad \xi\in[0,\xi_*],
\end{equation}
which are also solutions to \eqref{SSODE} with $\beta=\beta_2$, and adapt an optimal barrier argument from \cite{YeYin} (see also \cite[Lemma~4.12]{ILS22}). Owing to the monotonicity of $F_1$ and $F_2$ on $[0,X]$, we first note that
$$
\min\limits_{\xi\in[0,\xi_*]}G_{\lambda}(\xi)=G_{\lambda}(\xi_*)=\lambda^mF_2(\lambda^{-(m-1)/2}\xi_*)\geq\lambda^mF_2(\xi_*),
$$
whence
$$
\lim\limits_{\lambda\to\infty}\min\limits_{\xi\in[0,\xi_*]}G_{\lambda}(\xi)= \infty,
$$
while $F_1(\xi)\leq1$ for $\xi\in[0,\xi_*]$. Consequently, the optimal parameter
\begin{equation}\label{optimal.par}
\lambda_0:=\inf\{\lambda\geq1: G_{\lambda}(\xi)>F_1(\xi), \ \xi\in[0,\xi_*]\}
\end{equation}
is well defined and finite. Since $F_2<F_1$ on $(0,\xi_*)$, we also deduce that $\lambda_0>1$. The definition of $\lambda_0$ guarantees that there exists $\eta\in[0,\xi_*]$ such that
\begin{equation}\label{interm1}
G_{\lambda_0}(\eta)=F_1(\eta), \qquad G_{\lambda_0}\geq F_1 \ {\rm in} \ [0,\xi_*].
\end{equation}
On the one hand, we infer from the monotonicity of $F_2$ and the property $\lambda_0>1$ that
$$
F_1(\xi_*)=F_2(\xi_*)<\lambda_0^mF_2(\xi_*)<\lambda_0^mF_2(\lambda_0^{-(m-1)/2}\xi_*)=G_{\lambda_0}(\xi_*),
$$
which rules out the possibility that $\eta=\xi_*$. On the other hand,
$$
G_{\lambda_0}(0)=\lambda_0^mF_2(0)=\lambda_0^m>1=F_1(0),
$$
so that $\eta>0$. Consequently, $\eta\in(0,\xi_*)$ and we derive from~\eqref{interm1} that $G_{\lambda_0}-F_1$ attains a strict minimum at $\xi=\eta$, which, together with the definition of $\eta$, implies that
\begin{equation}\label{interm2}
G_{\lambda_0}(\eta)=F_1(\eta), \qquad G_{\lambda_0}'(\eta)=F_1'(\eta), \qquad G_{\lambda_0}''(\eta)\geq F_1''(\eta).
\end{equation}
Since both $G_{\lambda_0}$ and $F_1$ are solutions to~\eqref{SSODE} with parameters $\beta_2$ and $\beta_1$, respectively, we infer from~\eqref{interm2} that
\begin{equation*}
\begin{split}
0&=G_{\lambda_0}''(\eta)+\frac{N-1}{\eta}G_{\lambda_0}'(\eta)+\frac{2\beta_2}{m-1}G_{\lambda_0}^{1/m}(\eta)-\beta_2\eta \Big( G_{\lambda_0}^{1/m} \Big)'(\eta)-\eta^{\sigma}G_{\lambda_0}^{q/m}(\eta)\\
&\geq F_1''(\eta)+\frac{N-1}{\eta}F_1'(\eta)+\frac{2\beta_2}{m-1}F_1^{1/m}(\eta)-\beta_2\eta \Big(F_1^{1/m}\Big)'(\eta)-\eta^{\sigma}F_1^{q/m}(\eta)\\
&=-\frac{2\beta_1}{m-1}F_1^{1/m}(\eta)+\beta_1\frac{\eta}{m}F_1^{(1-m)/m}(\eta)F_1'(\eta)+\frac{2\beta_2}{m-1}F_1^{1/m}(\eta)-\beta_2\frac{\eta}{m}F_1^{(1-m)/m}(\eta)F_1'(\eta)\\
&=(\beta_2-\beta_1)F_1^{(1-m)/m}(\eta)\left[\frac{2}{m-1}F_1(\eta)-\frac{\eta}{m}F_1'(\eta)\right]>0,
\end{split}
\end{equation*}
which leads to a contradiction. We have thus established that $F_2<F_1$ on $(0,X)$ and the proof is complete due to the arbitrary choice of $X\in(0,\xi_1(\beta_2))\cap(0,\xi_1(\beta_1))$.
\end{proof}

Let us remark that, in contrast to the range $\sigma>\sigma_c$ studied in \cite[Section~3]{ILS22}, in our case the profiles $f(\cdot;\beta)$ are ordered in a decreasing way with respect to the shooting parameter~$\beta$.

\subsection{Interface behavior}\label{subsec.interf}

The goal of this section is deriving the local behavior near the interface point $\xi_0(\beta)$ for profiles $f(\cdot;\beta)$ with $\beta\in\mathcal{B}$. We begin with a formal calculation. Let us drop for simplicity $\beta$ from the notation and assume that, at the interface, we have
$$
f(\xi)\sim A(\xi_0-\xi)^{\theta}, \qquad f'(\xi)\sim-A\theta(\xi_0-\xi)^{\theta-1}, \qquad {\rm as} \ \xi\to\xi_0=\xi_0(\beta),
$$
for some $A>0$ and $\theta>0$ to be determined. We also obtain formally that
$$
(f^m)'(\xi)\sim-m\theta A^m(\xi_0-\xi)^{m\theta-1}, \qquad (f^m)''(\xi)\sim m\theta(m\theta-1)A^m(\xi_0-\xi)^{m\theta-2},
$$
both equivalences holding true as $\xi\to\xi_0$. Inserting this ansatz in~\eqref{ODE} gives, as $\xi\to\xi_0$,
\begin{equation*}
\begin{split}
&m\theta(m\theta-1)A^m(\xi_0-\xi)^{m\theta-2}-\frac{N-1}{\xi_0}m\theta A^m(\xi_0-\xi)^{m\theta-1}\\
&+\beta\xi_0A\theta(\xi_0-\xi)^{\theta-1}+\frac{2\beta}{m-1}A(\xi_0-\xi)^{\theta}-A^q\xi_0^{\sigma}(\xi_0-\xi)^{q\theta}=0.
\end{split}
\end{equation*}
We thus have four possibilities of balancing the dominating powers.

$\bullet$ $m\theta-2=\theta-1< q\theta$. This implies $\theta=1/(m-1)$, but in this case $m\theta-1=\theta>0$ and thus this choice leads to $A=0$.

$\bullet$ $\theta-1=q\theta < m\theta-2$. This implies $\theta=1/(1-q)$ and $m\theta-2>q\theta$ leads straightforwardly to $m+q>2$.

$\bullet$ $m\theta-2=q\theta<\theta-1$. This implies $\theta=2/(m-q)$ and the inequality $\theta-1>q\theta$ easily gives $m+q<2$.

$\bullet$ $m\theta-2=q\theta=\theta-1$. This implies that $\theta=1/(m-1)=1/(1-q)$ and $m+q=2$.

Looking now at the constant $A$ in front of the previous ansatz, we find the following three cases:

\medskip

\noindent \textbf{Case 1. $m+q>2$}. According to the formal calculation, we expect $\theta=1/(1-q)$ and then $\beta\xi_0 A\theta=A^q\xi_0^{\sigma}$, which leads to
\begin{equation}\label{interm3}
A^{1-q}=\frac{(1-q)\xi_0^{\sigma-1}}{\beta}.
\end{equation}

\medskip

\noindent \textbf{Case 2. $m+q=2$}. We expect $\theta=1/(1-q)=2/(m-q)$ and
$$
m\theta(m\theta-1)A^m+\beta\xi_0 A\theta-A^q\xi_0^{\sigma}=0;
$$
that is, $A=A_*$ with $A_*$ being the unique positive solution to
\begin{equation*}
	\frac{m(m+q-1)}{(1-q)^2}A_{*}^{m-q}+\frac{\beta\xi_0}{1-q}A_{*}^{1-q} -\xi_0^{\sigma}=0.
\end{equation*}
Since $m+q=2$ and $\sigma=2$ in that case, the above equation simplifies to
\begin{equation}\label{interm4}
\frac{m}{(1-q)^2}A_{*}^{m-q}+\frac{\beta\xi_0}{1-q}A_{*}^{(m-q)/2}-\xi_0^2=0.
\end{equation}

\medskip

\noindent \textbf{Case 3. $m+q<2$}. We expect $\theta=2/(m-q)$ and $m\theta(m\theta-1)A^m=A^q\xi_0^{\sigma}$, hence
\begin{equation}\label{interm5}
A^{m-q}=\frac{(m-q)^2}{2m(m+q)}\xi_0^{\sigma}.
\end{equation}

\medskip

In order to prove in a rigorous way all these estimates near the interface, we proceed as in \cite{ILS22}. We start with some general upper bounds at the interface, but omit the proof, as it is totally similar to that of \cite[Lemma~4.7]{ILS22}.

\begin{lemma}\label{lem.upper.interf}
Assume that $\beta\in\mathcal{B}$ and set $f=f(\cdot;\beta)$ and $\xi_0=\xi_0(\beta)$. Then
\begin{equation}\label{interf1}
|(f^{m-q})'(\xi)|\leq 2^{N-1}\xi_0^{\sigma}(\xi_0-\xi), \qquad \xi\in\left(\frac{\xi_0}{2},\xi_0\right),
\end{equation}
and
\begin{equation}\label{interf2}
f(\xi)\leq\beta^{q-1}\xi_0^{(\sigma-1)/(1-q)}(\xi_0-\xi)^{1/(1-q)}, \qquad \xi\in\left(\frac{\xi_0}{2},\xi_0\right).
\end{equation}
Moreover, there exists $C_1>0$ depending only on $N$, $m$ and $q$ such that
\begin{equation}\label{interf3}
f(\xi)\leq C_1\xi_0^{\sigma/(m-q)}(\xi_0-\xi)^{2/(m-q)}, \qquad \xi\in\left(\frac{\xi_0}{2},\xi_0\right).
\end{equation}
\end{lemma}

The following consequences of Lemma~\ref{lem.upper.interf} are drawn in the same way as in \cite[Lemmas~4.8 and~4.9]{ILS22}.

\begin{corollary}\label{cor.interf}
Let $\beta\in\mathcal{B}$ and set $f=f(\cdot;\beta)$ and $\xi_0=\xi_0(\beta)$. Then
$$
\limsup\limits_{\xi\to\xi_0}\left(f^{(m-q)/2}\right)'(\xi)>-\infty.
$$
In addition, if $m+q>2$ then
$$
\limsup\limits_{\xi\to\xi_0}\left(f^{m-1}\right)'(\xi)=0.
$$
\end{corollary}

The estimates given in Corollary~\ref{cor.interf} allow us to proceed as in \cite[Propositions~4.10 and~4.11]{ILS22} in order to identify the precise algebraic rate at which $f(\cdot;\beta)$ vanishes at the interface, which depends on the sign of $m+q-2$ as follows.

\begin{proposition}\label{prop.interf}
Let $\beta\in\mathcal{B}$ and set $f=f(\cdot;\beta)$ and $\xi_0=\xi_0(\beta)$.

\medskip

(a) If $m+q<2$, then, as $\xi\to\xi_0$,
\begin{equation}\label{int.small}
f(\xi)=K_1\xi_0^{\sigma/(m-q)}(\xi_0-\xi)^{2/(m-q)}+o((\xi_0-\xi)^{2/(m-q)}),
\end{equation}
where
$$
K_1:=\left[\frac{m-q}{\sqrt{2m(m+q)}}\right]^{2/(m-q)}.
$$

\medskip

(b) If $m+q=2$, then $\sigma=2$ and, as $\xi\to\xi_0$,
\begin{equation}\label{int.equal}
f(\xi)=K_1\xi_0^{2/(m-q)}K_2(\beta)(\xi_0-\xi)^{2/(m-q)}+o((\xi_0-\xi)^{2/(m-q)}),
\end{equation}
where $K_1$ is defined in part (a) and
$$
K_2(\beta):=\left[\sqrt{1+\frac{\beta^2}{4m}}-\frac{\beta}{2\sqrt{m}}\right]^{2/(m-q)}.
$$

\medskip

(c) If $m+q>2$, then, as $\xi\to\xi_0$,
\begin{equation}\label{int.large}
f(\xi)=K_3(\beta)\xi_0^{(\sigma-1)/(1-q)}(\xi_0-\xi)^{1/(1-q)}+o((\xi_0-\xi)^{1/(1-q)}),
\end{equation}
where
$$
K_3(\beta):=\left[\frac{1-q}{\beta}\right]^{1/(1-q)}.
$$
\end{proposition}

Let us notice here that the values of $K_1$, $K_2(\beta)$ and $K_3(\beta)$ in~\eqref{int.small}, \eqref{int.equal} and~\eqref{int.large} correspond to the values of $A$ obtained through the formal deduction in~\eqref{interm5}, \eqref{interm4} and~\eqref{interm3}, respectively. It is now worth pointing out that there is no explicit dependence on $\beta$ in the behavior~\eqref{int.small} when $m+q<2$. This is why we need to perform some rather serious extra work in order to identify the second order of the expansion at the interface when $m+q\in(1,2)$, as formal computations (which are rather tedious and we do not give here) reveal that $\beta$ shows up in an explicit way in this next order, a feature that will be very helpful in the proof of the uniqueness issue. More precisely, we have the following asymptotic expansions.

\begin{proposition}\label{prop.order2}
Let $m+q<2$, $\beta\in\mathcal{B}$ and set $f=f(\cdot;\beta)$ and $\xi_0=\xi_0(\beta)$. Then, as $\xi\to\xi_0$,
\begin{equation}\label{int.small2}
\begin{split}
f(\xi)&=K_1\xi_0^{\sigma/(m-q)}(\xi_0-\xi)^{2/(m-q)}\\
& \qquad - K_0(\beta)\xi_0^{(\sigma+m+q-2)/(m-q)}(\xi_0-\xi)^{(4-m-q)/(m-q)}\\
& \qquad +o((\xi_0-\xi)^{(4-m-q)/(m-q)}),
\end{split}
\end{equation}
where $K_1$ is defined in \eqref{int.small} and
\begin{equation}\label{interm10}
K_0(\beta):=\frac{(m-q)\beta K_1^{2-m}}{m(1-q)(m+q+2)}.
\end{equation}
\end{proposition}

\begin{proof}
As in the proof of \cite[Proposition~4.10]{ILS22}, we introduce the new dependent variables
\begin{equation}\label{change.small}
\begin{split}
\mathcal{X}(\xi) & := \sqrt{m} \xi^{-(\sigma+2)/2} f^{(m-q)/2}(\xi), \\
\mathcal{Y}(\xi) & := \sqrt{m} \xi^{-\sigma/2} f^{(m-q-2)/2}(\xi) f'(\xi), \\
\mathcal{Z}(\xi) & := \frac{\alpha}{\sqrt{m}} \xi^{(2-\sigma)/2} f^{(2-m-q)/2}(\xi),
\end{split}
\end{equation}
as well as a new independent variable $\eta$ via the integral representation
\begin{equation}\label{PSvar.low}
\eta(\xi) := \frac{1}{\sqrt{m}} \int_0^\xi f^{(q-m)/2}(\xi_*) \xi_*^{\sigma/2} \,d\xi_*, \qquad \xi\in [0,\xi_0).
\end{equation}
Introducing $(X,Y,Z)$ defined by $(\mathcal{X},\mathcal{Y},\mathcal{Z})=(X\circ\eta,Y\circ\eta,Z\circ\eta)$, we see that $(X,Y,Z)$ solves the quadratic autonomous dynamical system
\begin{equation}\label{PSSyst.low}
\left\{
\begin{split}
	\dot{X}&= X \left[ \frac{m-q}{2} Y - \frac{\sigma+2}{2} X\right]\\[1mm]
	\dot{Y}&= - \frac{m+q}{2} Y^2 - \left( N-1+\frac{\sigma}{2} \right) XY - XZ + \frac{m-1}{2} YZ + 1\\[1mm]
	\dot{Z}&= Z \left[ \frac{2-m-q}{2}Y + \frac{2-\sigma}{2}X \right].
\end{split}
\right.
\end{equation}
Observe that, owing to~\eqref{int.small},
\begin{equation*}
	\lim_{\xi\to\xi_0} \eta(\xi) = \infty,
\end{equation*}
so that studying the behavior of $(\mathcal{X},\mathcal{Y},\mathcal{Z})(\xi)$ as $\xi\to\xi_0$ amounts to that of $(X,Y,Z)(\eta)$ as $\eta\to\infty$. Furthermore, we argue as in \cite[Proposition~4.10]{ILS22} to deduce from~\eqref{int.small} and Corollary~\ref{cor.interf} that
\begin{equation*}
	(X,Y,Z)(\eta) \in (0,\infty)\times (-\infty,0) \times (0,\infty), \qquad \eta>0,
\end{equation*}
and
\begin{equation*}
	\lim_{\eta\to\infty} (X,Y,Z)(\eta) = \left( 0 , -\sqrt{\frac{2}{m-q}} , 0 \right).
\end{equation*}
We are thus interested in the behavior near the critical point $(0,-\sqrt{2/(m+q)},0)$. We translate this point to the origin of coordinates by setting
\begin{equation}\label{trans.low}
W=Y+\sqrt{\frac{2}{m+q}}.
\end{equation}
We then find by direct calculation that the system \eqref{PSSyst.low} becomes
\begin{equation}\label{PSSyst.low2}
\left\{
\begin{split}
	\dot{X}&= -\frac{m-q}{\sqrt{2(m+q)}}X+\frac{m-q}{2}XW - \frac{\sigma+2}{2}X^2\\[1mm]
	\dot{W}&= \left(N-1+\frac{\sigma}{2}\right) \sqrt{\frac{2}{m+q}} X + \sqrt{2(m+q)}W - \frac{m-1}{\sqrt{2(m+q)}} Z\\[1mm]
           & \qquad -\left(N-1+\frac{\sigma}{2}\right)XW-XZ-\frac{m+q}{2}W^2 + \frac{m-1}{2}WZ\\[1mm]
	\dot{Z}&= -\frac{2-m-q}{\sqrt{2(m+q)}}Z+\frac{2-m-q}{2} WZ +\frac{2-\sigma}{2}XZ.
\end{split}
\right.
\end{equation}
Introducing $\mathbf{F}(\mathbf{V}) = (F_1,F_2,F_3)(\mathbf{V})$ defined for $\mathbf{V}=(V_1,V_2,V_3)\in\mathbb{R}^3$ by
\begin{align*}
	F_1(\mathbf{V}) & := -\frac{m-q}{\sqrt{2(m+q)}}V_1 +\frac{m-q}{2} V_1 V_2 - \frac{\sigma+2}{2} V_1^2\\[1mm]
	F_2(\mathbf{V}) & := \left(N-1+\frac{\sigma}{2}\right) \sqrt{\frac{2}{m+q}} V_1 + \sqrt{2(m+q)}V_2 - \frac{m-1}{\sqrt{2(m+q)}} V_3\\[1mm]
	&-\left(N-1+\frac{\sigma}{2}\right) V_1 V_2 - V_1 V_3 - \frac{m+q}{2}V_2^2 + \frac{m-1}{2}V_2 V_3\\[1mm]
	F_3(\mathbf{V}) & := -\frac{2-m-q}{\sqrt{2(m+q)}} V_3 + \frac{2-m-q}{2} V_2 V_3 +\frac{2-\sigma}{2} V_1 V_3,
\end{align*}
and denoting the semiflow associated with the dynamical system
\begin{equation}
	\dot{\mathbf{V}}(\eta) = \mathbf{F}(\mathbf{V(\eta)}), \qquad \eta>0, \qquad \mathbf{V}(0) = \mathbf{V}_0 \in \mathbb{R}^3, \label{z1}
\end{equation}
by $\varphi(\cdot;\mathbf{V}_0)$, we deduce from~\eqref{PSSyst.low2} that $\mathbf{V}_* := (X,W,Z) = \varphi(\cdot;\mathbf{V}_*(0))$ is defined on $[0,\infty)$ with
\begin{equation}
	\lim_{\eta\to\infty} \mathbf{V}_*(\eta) = 0. \label{z2}
\end{equation}
The matrix associated with the linearization of the system~\eqref{z1} at the origin is
$$
\mathcal{M}=\sqrt{\frac{2}{m+q}}
\begin{pmatrix}
	\displaystyle{-\frac{m-q}{2}} & 0 & 0 \\
	\displaystyle{N-1+\frac{\sigma}{2}} & m+q & \displaystyle{-\frac{m-1}{2}} \\
	0 & 0 & \displaystyle{-\frac{2-m-q}{2}}
\end{pmatrix}
$$
having three distinct eigenvalues
$$
\lambda_1=-\frac{m-q}{\sqrt{2(m+q)}}, \qquad \lambda_2=\sqrt{2(m+q)}, \qquad \lambda_3=-\frac{2-m-q}{\sqrt{2(m+q)}},
$$
with corresponding eigenvectors (not normalized)
$$
E_1=\left(1,-\frac{2(N-1)+\sigma}{3m+q},0\right), \qquad E_2=(0,1,0), \qquad E_3=\left(0,\frac{m-1}{2+m+q},1\right).
$$
Then $\mathbf{0}$ is an hyperbolic point of $\varphi$ and has a two-dimensional stable manifold $\mathcal{W}_s(\mathbf{0})$. According to the proof of the stable manifold theorem (see for example \cite[Theorem~19.11]{Amann}), there is an open neighborhood $\mathcal{V}$ of zero in $\real^3$, an open neighborhood $\mathcal{V}_0$ of zero in $\real^2$ and a $C^2$-smooth function $h:\mathcal{V}_0\to\real$ such that $h(0,0)=\partial_x h(0,0)=\partial_z h(0,0)=0$ and the local stable manifold
$$
\mathcal{W}_s^{\mathcal{V}}(\mathbf{0}) := \left\{ \mathbf{V}_0\in \mathcal{W}_s(\mathbf{0})\ :\ \varphi(\eta;\mathbf{V}_0)\in \mathcal{V} \; {\rm for \ all} \; \eta\ge 0 \right\}
$$
satisfies
$$
\mathcal{W}_s^{\mathcal{V}}(\mathbf{0})\subseteq\{xE_1+h(x,z)E_2+zE_3:(x,z)\in\mathcal{V}_0\},
$$
its tangent space at $\mathbf{0}$ being $\mathbb{R} E_1 \oplus \mathbb{R} E_3$. Since $\{\varphi(\eta;\mathbf{V}_*(0))\ : \ \eta\geq \eta_0\}$ is included in $\mathcal{W}_s(\mathbf{0})\cap \mathcal{V}$ for $\eta_0$ large enough by~\eqref{z2}, we conclude that $\varphi(\eta;\mathbf{V}_*(0))$ belongs to $\mathcal{W}_s^{\mathcal{V}}(\mathbf{0})$ for $\eta\ge \eta_0$. Consequently, there are functions $(\overline{x},\overline{z}):[\eta_0,\infty)\to\mathcal{V}_0$ such that
$$
(X,W,Z)(\eta) = \varphi(\eta;\mathbf{V}_*(0)) = \overline{x}(\eta)E_1 + h(\overline{x}(\eta),\overline{z}(\eta))E_2 + \overline{z}(\eta)E_3
$$
for $\eta\geq \eta_0$. In fact, $\overline{x}(\eta)=X(\eta)$, $\overline{z}(\eta)=Z(\eta)$ and
\begin{equation}\label{interm6}
W(\eta)=-\frac{2(N-1)+\sigma}{3m+q}X(\eta)+\frac{m-1}{2+m+q}Z(\eta)+h(X(\eta),Z(\eta)).
\end{equation}
Let us notice from \eqref{change.small} that
$$
\mathcal{Z}(\xi)=\alpha m^{(q-1)/(m-q)}\mathcal{X}^{(2-m-q)/(m-q)}(\xi),
$$
which implies that $X(\eta)=o(Z(\eta))$ as $\eta\to\infty$, since $(2-m-q)/(m-q)<1$. Recalling also that $h$ is $C^2$-smooth with $h(0,0)=\partial_x h(0,0)=\partial_z h(0,0)=0$, we infer from~\eqref{interm6} that
$$
W(\eta)=\frac{m-1}{2+m+q}Z(\eta)+o(Z(\eta)) \qquad {\rm as} \ \eta\to\infty,
$$
or equivalently, undoing the change of variable~\eqref{PSvar.low} and the translation~\eqref{trans.low}, we get as $\xi\to\xi_0$,
\begin{equation}\label{interm7}
\mathcal{Y}(\xi)=-\sqrt{\frac{2}{m+q}}+\frac{m-1}{2+m+q}\mathcal{Z}(\xi)+o(\mathcal{Z}(\xi)).
\end{equation}
Moreover, we readily infer from the already obtained local behavior \eqref{int.small} and the definition of $\mathcal{Z}$ in~\eqref{change.small} that, as $\xi\to\xi_0$,
$$
\mathcal{Z}(\xi)\sim\frac{\alpha}{\sqrt{m}}K_1^{(2-m-q)/2}\xi_0^{(2-\sigma)/2+\sigma(2-m-q)/2(m-q)}(\xi_0-\xi)^{(2-m-q)/(m-q)}.
$$
Inserting the previous expansion into \eqref{interm7} and recalling the definition of $\mathcal{Y}$ in~\eqref{change.small}, we find
\begin{equation*}
\begin{split}
\frac{2\sqrt{m}}{m-q} & \xi^{-\sigma/2} \left(f^{(m-q)/2}\right)'(\xi) = -\sqrt{\frac{2}{m+q}}\\
& + \frac{\alpha (m-1) K_1^{(2-m-q)/2}}{(2+m+q)\sqrt{m}} \xi_0^{(2-\sigma)/2+\sigma(2-m-q)/2(m-q)}(\xi_0-\xi)^{(2-m-q)/(m-q)}\\
&+o\left((\xi_0-\xi)^{(2-m-q)/(m-q)}\right),
\end{split}
\end{equation*}
which leads to, since $\alpha=2\beta/(m-1)$,
\begin{equation*}
\begin{split}
& \hspace{-0.25cm} \left(f^{(m-q)/2}\right)'(\xi) \\
&=-K_1^{(m-q)/2}\xi_0^{\sigma/2}\left(1-\frac{\xi_0-\xi}{\xi_0}\right)^{\sigma/2}\\
& \qquad + (1-q) K_0(\beta) K_1^{(m-q-2)/2}\xi_0^{[2(m-q)+\sigma(2-m-q)]/2(m-q)}(\xi_0-\xi)^{(2-m-q)/(m-q)}\\
& \qquad\qquad\qquad  \times \left(1-\frac{\xi_0-\xi}{\xi_0}\right)^{\sigma/2} +o\left((\xi_0-\xi)^{(2-m-q)/(m-q)}\right)\\
&=-K_1^{(m-q)/2}\xi_0^{\sigma/2}\left(1-\frac{\sigma(\xi_0-\xi)}{2\xi_0}\right)\\
& \qquad + (1-q) K_0(\beta) K_1^{(m-q-2)/2}\xi_0^{[2(m-q)+\sigma(2-m-q)]/2(m-q)}(\xi_0-\xi)^{(2-m-q)/(m-q)}\\
& \qquad +o\left((\xi_0-\xi)^{(2-m-q)/(m-q)}\right).
\end{split}
\end{equation*}
Recalling that $(2-m-q)/(m-q)<1$, we end up with
\begin{equation}\label{interm8}
\begin{split}
& \hspace{-0.25cm} \left(f^{(m-q)/2}\right)'(\xi) = -K_1^{(m-q)/2} \xi_0^{\sigma/2} \\
& \qquad + (1-q) K_0(\beta) K_1^{(m-q-2)/2}\xi_0^{[2(m-q)+\sigma(2-m-q)]/2(m-q)} (\xi_0-\xi)^{(2-m-q)/(m-q)} \\
& \qquad + o\left((\xi_0-\xi)^{(2-m-q)/(m-q)}\right).
\end{split}
\end{equation}
Integrating \eqref{interm8} over $(\xi,\xi_0)$ and then taking powers $2/(m-q)$ give
\begin{equation*}
\begin{split}
f(\xi)&=K_1\xi_0^{\sigma/(m-q)}(\xi_0-\xi)^{2/(m-q)}\\
& \qquad \times \left[1 -\frac{(m-q) K_0(\beta)}{2 K_1} \xi_0^{(m+q-2)/(m-q)} (\xi_0-\xi)^{(2-m-q)/(m-q)}\right.\\
& \qquad\qquad\qquad  \left.+o\left((\xi_0-\xi)^{(2-m-q)/(m-q)}\right)\right]^{2/(m-q)}\\
&=K_1\xi_0^{\sigma/(m-q)}(\xi_0-\xi)^{2/(m-q)} - K_0(\beta) \xi_0^{(\sigma+m+q-2)/(m-q)}(\xi_0-\xi)^{(4-m-q)/(m-q)}\\
& \qquad +o\left((\xi_0-\xi)^{(4-m-q)/(m-q)}\right),
\end{split}
\end{equation*}
as stated.
\end{proof}

\subsection{Uniqueness}\label{subsec.uniq}

We are now ready to complete the proof of Theorem~\ref{th.1} by showing that the set $\mathcal{B}$ contains at most one element. Taking into account the previous preparations, this proof borrows ideas from the analogous one in \cite[Section~4.4]{ILS22}.

\medskip

\begin{proof}[Proof of Theorem~\ref{th.1}: uniqueness]
Assume for contradiction that there are $\beta_1\in\mathcal{B}$ and $\beta_2\in\mathcal{B}$ such that $0<\beta_1<\beta_2<\infty$. By Lemma~\ref{lem.B}, we have $\xi_0(\beta_1)=\xi_1(\beta_1)$ and $\xi_0(\beta_2)=\xi_1(\beta_2)$, so that Lemma~\ref{lem.monot} implies that $f_1(\xi)>f_2(\xi)$ and $F_1(\xi)>F_2(\xi)$ for any $\xi\in \Big(0,\min\{\xi_0(\beta_1),\xi_0(\beta_2)\}\Big)$, with $f_i:=f(\cdot;\beta_i)$ and $F_i:=f_i^m$ for $i=1,2$. In particular, $\xi_0(\beta_2)<\xi_0(\beta_1)$.

As in the proof of Lemma~\ref{lem.monot}, see~\eqref{resc2}-\eqref{optimal.par}, we introduce the rescaled version $G_\lambda$ of $F_2$ defined by
\begin{equation}
	G_\lambda(\xi) := \lambda^m F_2\big(\lambda^{-(m-1)/2}\xi\big), \qquad \xi\in [0,\infty), \quad \lambda\ge 1, \label{resc2b}
\end{equation}
recalling that $F_2$ is well-defined on $[0,\infty)$ by Lemma~\ref{lem.ext}, and define the optimal parameter
\begin{equation}
	\lambda_0 := \inf\left\{ \lambda\ge 1\ :\ G_\lambda(\xi)>F_1(\xi), \ \xi\in [0,\xi_0(\beta_1)]\right\} \in (1,\infty), \label{optimal.parb}
\end{equation}
its existence being ensured by the fact that
\begin{equation*}
\begin{split}
\lim\limits_{\lambda\to\infty}\min\limits_{\xi\in[0,\xi_0(\beta_1)]}G_{\lambda}(\xi)&=\lim\limits_{\lambda\to\infty}G_{\lambda}(\xi_0(\beta_1))=
\lim\limits_{\lambda\to\infty}\lambda^mF_2(\lambda^{-(m-1)/2}\xi_0(\beta_1))\\
&\geq\lim\limits_{\lambda\to\infty}\lambda^mF_2\left(\frac{\xi_0(\beta_2)}{2}\right)=\infty.
\end{split}
\end{equation*}
According to the definition of $\lambda_0$ in \eqref{optimal.parb} and the compactness of the interval $[0,\xi_0(\beta_1)]$, we deduce that there is $\eta\in[0,\xi_0(\beta_1)]$ such that $F_1(\eta)=G_{\lambda_0}(\eta)$ and $F_1\leq G_{\lambda_0}$ on $[0,\xi_0(\beta_1)]$. Arguments very similar to the ones employed in the proof of Lemma~\ref{lem.monot}, along with Lemma~\ref{lem.ext}, then discard the possibility that either $\eta=0$ or $\eta\in(0,\xi_0(\beta_1))$, thus showing that $\eta=\xi_0(\beta_1)$. Consequently,
\begin{equation}\label{interm9}
F_1(\xi_0(\beta_1))=0=G_{\lambda_0}(\xi_0(\beta_1)), \qquad 0<F_1(\xi)<G_{\lambda_0}(\xi), \qquad \xi\in[0,\xi_0(\beta_1)),
\end{equation}
and we also obtain the following equality implied by the equality of the supports in~\eqref{interm9} and the rescaling~\eqref{resc2b}
\begin{equation}\label{supports}
\xi_0(\beta_1)=\lambda_0^{(m-1)/2}\xi_0(\beta_2).
\end{equation}
We now split the analysis into the three cases already set apart at the beginning of Section~\ref{subsec.interf}, according to the sign of $m+q-2$.

\medskip

\noindent \textbf{Case~1. $m+q<2$}. We recall that, in this case, Proposition~\ref{prop.order2} gives
\begin{equation}
\begin{split}
f_i(\xi)&=K_1\xi_0(\beta_i)^{\sigma/(m-q)}(\xi_0(\beta_i)-\xi)^{2/(m-q)}\\
& \qquad -K_0(\beta_i)\xi_0(\beta_i)^{(\sigma+m+q-2)/(m-q)}(\xi_0(\beta_i)-\xi)^{(4-m-q)/(m-q)}\\
& \qquad +o((\xi_0(\beta_i)-\xi)^{(4-m-q)/(m-q)}),
\end{split}\label{zeus}
\end{equation}
as $\xi\to\xi_0(\beta_i)$, $i=1,2$. In order to simplify the calculations, we can work at the level of $f_i$ by noticing that the rescaling~\eqref{resc2b} reduces to
\begin{equation}\label{resc3}
g_{\lambda_0}(\xi):=G_{\lambda_0}^{1/m}(\xi)=\lambda_0f_2\left(\lambda_0^{-(m-1)/2}\xi\right).
\end{equation}
We thus infer from~\eqref{zeus} and~\eqref{resc3} that
\begin{equation*}
\begin{split}
g_{\lambda_0}(\xi)&=\lambda_0K_1\xi_0(\beta_2)^{\sigma/(m-q)}\left(\xi_0(\beta_2)-\lambda_0^{-(m-1)/2}\xi\right)^{2/(m-q)}\\
& \qquad -K_0(\beta_2)\lambda_0\xi_0(\beta_2)^{(\sigma+m+q-2)/(m-q)}(\xi_0(\beta_2)-\lambda_0^{-(m-1)/2}\xi)^{(4-m-q)/(m-q)}\\
& \qquad + o\left(\left(\xi_0(\beta_2)-\lambda_0^{-(m-1)/2}\xi\right)^{(4-m-q)/(m-q)}\right)\\
&=\lambda_0K_1\left(\lambda_0^{-(m-1)/2}\xi_0(\beta_1)\right)^{\sigma/(m-q)}\lambda_0^{-(m-1)/(m-q)}(\xi_0(\beta_1)-\xi)^{2/(m-q)}\\
& \qquad -K_0(\beta_2)\lambda_0\left(\lambda_0^{-(m-1)/2}\xi_0(\beta_1)\right)^{(\sigma+m+q-2)/(m-q)}\lambda_0^{-(m-1)(4-m-q)/2(m-q)}\\
& \qquad \qquad \times(\xi_0(\beta_1)-\xi)^{(4-m-q)/(m-q)}+o\left((\xi_0(\beta_1)-\xi)^{(4-m-q)/(m-q)}\right).
\end{split}
\end{equation*}
Noticing that the powers of $\lambda_0$ appearing in the (rather tedious) previous calculations cancel out due to the precise value of $\sigma$ given in \eqref{range.exp}, we further obtain
\begin{equation*}
\begin{split}
g_{\lambda_0}(\xi)&=K_1\xi_0(\beta_1)^{\sigma/(m-q)}(\xi_0(\beta_1)-\xi)^{2/(m-q)}\\
& \qquad -K_0(\beta_2)\xi_0(\beta_1)^{(\sigma+m+q-2)/(m-q)}(\xi_0(\beta_1)-\xi)^{(4-m-q)/(m-q)}\\
& \qquad +o\left((\xi_0(\beta_1)-\xi)^{(4-m-q)/(m-q)}\right)\\
&=f_1(\xi)+(K_0(\beta_1)-K_0(\beta_2))\xi_0(\beta_1)^{(\sigma+m+q-2)/(m-q)}(\xi_0(\beta_1)-\xi)^{(4-m-q)/(m-q)}\\
&\qquad +o\left((\xi_0(\beta_1)-\xi)^{(4-m-q)/(m-q)}\right).
\end{split}
\end{equation*}
Since $\beta_1<\beta_2$, we deduce from~\eqref{interm10} that $K_0(\beta_1)<K_0(\beta_2)$. Thus $g_{\lambda_0}(\xi)<f_1(\xi)$ in a left neighborhood of $\xi_0(\beta_1)$, whence (by raising to power $m$) $G_{\lambda_0}(\xi)<F_1(\xi)$ in the same left neighborhood of $\xi_0(\beta_1)$, and we have reached a contradiction to \eqref{interm9}.

\medskip

\noindent \textbf{Case~2. $m+q=2$}. In this case,  Proposition~\ref{prop.interf}~(b) gives
\begin{equation*}
	F_i(\xi)= K_1^m \xi_0(\beta_i)^{2m/(m-q)}K_2^m(\beta_i) (\xi_0(\beta_i)-\xi)^{2m/(m-q)}+ o\Big((\xi_0(\beta_i)-\xi)^{2m/(m-q)}\Big)
\end{equation*}
as $\xi\to\xi_0(\beta_i)$, $i=1,2$. We thus have
\begin{equation*}
\begin{split}
G_{\lambda_0}(\xi) & = \lambda_0^{m} K_1^m\xi_0(\beta_2)^{2m/(m-q)}K_2^m(\beta_2)\Big(\xi_0(\beta_2)-\lambda_0^{-(m-1)/2}\xi\Big)^{2m/(m-q)} \\
	& \qquad + o\Big((\xi_0(\beta_2)-\lambda_0^{-(m-1)/2}\xi)^{2m/(m-q)}\Big) \\
	& = \lambda_0^{m}K_1^m (\lambda_0^{-(m-1)/2}\xi_0(\beta_1))^{2m/(m-q)}K_2^m(\beta_2)\lambda_0^{-m(m-1)/(m-q)}(\xi_0(\beta_1)-\xi)^{2m/(m-q)} \\
	& \qquad + o\Big((\xi_0(\beta_1)- \xi)^{2m/(m-q)}\Big) \\
	& =K_1^m \xi_0(\beta_1)^{2m/(m-q)}K_2^m(\beta_2)(\xi_0(\beta_1)-\xi)^{2m/(m-q)} \\
	& \qquad + o\Big((\xi_0(\beta_1)- \xi)^{2m/(m-q)}\Big) \\
	& =\left[ \frac{ K_2(\beta_2)}{K_2(\beta_1)}\right]^m F_1(\xi)+o\Big((\xi_0(\beta_1)- \xi)^{2m/(m-q)}\Big) ,
\end{split}
\end{equation*}
the powers of $\lambda_0$ canceling out due to $m+q=2$. Noticing that we can write
$$
K_2(\beta)=\left[\sqrt{1+\frac{\beta^2}{4m}}+\frac{\beta}{2\sqrt{m}}\right]^{-2/(m-q)},
$$
we easily observe that $K_2$ is a decreasing function of $\beta$, thus $K_2(\beta_2)<K_2(\beta_1)$ since $\beta_2>\beta_1$. Therefore, $G_{\lambda_0}(\xi)<F_1(\xi)$ in a left neighborhood of $\xi_0(\beta_1)$, which contradicts \eqref{interm9}.

\medskip

\noindent \textbf{Case~3. $m+q>2$}. We recall that, in this case, Proposition~\ref{prop.interf}~(c) gives
$$
F_i(\xi)=K_3^m(\beta_i) \xi_0(\beta_i)^{m(\sigma-1)/(1-q)} (\xi_0(\beta_i)-\xi)^{m/(1-q)} + o((\xi_0(\beta_i)-\xi)^{m/(1-q)})
$$
as $\xi\to\xi_0(\beta_i)$, $i=1,2$. Using then the rescaling~\eqref{resc2b} and the identity~\eqref{supports}, we readily infer that
\begin{equation*}
\begin{split}
G_{\lambda_0}(\xi)&=\lambda_0^m K_3^m(\beta_2) \xi_0(\beta_2)^{m(\sigma-1)/(1-q)} \left(\xi_0(\beta_2) - \lambda_0^{-(m-1)/2}\xi\right)^{m/(1-q)}\\
&\qquad +o\left(\left(\xi_0(\beta_2)-\lambda_0^{-(m-1)/2}\xi\right)^{m/(1-q)}\right)\\
&=\lambda_0^m K_3^m(\beta_2) \left(\lambda_0^{-(m-1)/2}\xi_0(\beta_1)\right)^{m(\sigma-1)/(1-q)} \lambda_0^{-(m-1)m/2(1-q)}(\xi_0(\beta_1)-\xi)^{m/(1-q)}\\
&\qquad +o\left((\xi_0(\beta_1)-\xi)^{m/(1-q)}\right)\\
&=K_3^m(\beta_2) \xi_0(\beta_1)^{m(\sigma-1)/(1-q)} (\xi_0(\beta_1)-\xi)^{m/(1-q)} + o((\xi_0(\beta_1)-\xi)^{m/(1-q)})\\
&=\left[\frac{K_3(\beta_2)}{K_3(\beta_1)}\right]^m F_1(\xi)+o((\xi_0(\beta_1)-\xi)^{m/(1-q)}).
\end{split}
\end{equation*}
Since $K_3(\beta_2)<K_3(\beta_1)$ for $\beta_2>\beta_1$, we find that $G_{\lambda_0}(\xi)<F_1(\xi)$ in a left neighborhood of $\xi_0(\beta_1)$, which is again a contradiction to \eqref{interm9}.

\medskip

The previous contradictions imply that there cannot be two different values of the exponent $\beta$ in the set $\mathcal{B}$, completing the proof.
\end{proof}

\bigskip

\noindent \textbf{Acknowledgements} This work is partially supported by the Spanish project PID2020-115273GB-I00 and by the Grant RED2022-134301-T (Spain). Part of this work has been developed during visits of R. G. I. to Institut de Math\'ematiques de Toulouse and to Laboratoire de Math\'ematiques LAMA, Universit\'e de Savoie, and of Ph. L. to Universidad Rey Juan Carlos, and both authors thank these institutions for hospitality and support. The authors wish to thank Ariel S\'anchez (Universidad Rey Juan Carlos) for interesting comments and suggestions.

\bibliographystyle{plain}

\end{document}